\documentclass[11pt]{article}

\usepackage{amssymb,amsfonts,amsmath,cite}
\usepackage{enumerate}
\usepackage{tikz}
\usepackage{forloop}
\usepackage{verbatim}

\newcommand{\qed}{\hfill $\square$ \bigskip}

\newtheorem{theorem}{Theorem}[section]

\newtheorem{lemma}[theorem]{Lemma}

\newtheorem{conjecture-conclude}{Conjecture}
\newtheorem{problem-conclude}[conjecture-conclude]{Problem}

\newcommand\bonusspiral{} % just for safety
\def\bonusspiral[#1](#2)(#3:#4)(#5:#6)[#7]{% \bonusspiral[draw options](placement)(start angle:end angle)(start radius:final radius)[revolutions]
	\pgfmathsetmacro{\domain}{#4+#7*360}
	\pgfmathsetmacro{\growth}{180*(#6-#5)/(pi*(\domain-#3))}
	\draw [#1,
	shift={(#2)},
	domain=#3*pi/180:\domain*pi/180,
	variable=\t,
	smooth,
	samples=int(\domain/5)] plot ({\t r}: {#5+\growth*\t-\growth*#3*pi/180})
}

\begin{document}

\title{$S$-packing colorings of distance graphs with distance sets of cardinality $2$}

\author{
Bo\v{s}tjan Bre\v{s}ar$^{a,b}$  \and Jasmina Ferme$^{c,a}$ \and P\v{r}emysl Holub$^{d}$ \and Marko Jakovac$^{a,b}$ \and Petra Melicharov\' a$^{d}$
}

\date{\today}

\maketitle

\begin{center}
$^a$ Faculty of Natural Sciences and Mathematics, University of Maribor, Slovenia\\
\medskip

$^b$ Institute of Mathematics, Physics and Mechanics, Ljubljana, Slovenia\\
\medskip

$^c$ Faculty of Education, University of Maribor, Slovenia\\
\medskip

$^d$ University of West Bohemia, Plze\v n, Czech Republic
\\
%\medskip

%$^e$
\end{center}

\begin{abstract}
For a non-decreasing sequence $S=(s_1,s_2,\ldots)$ of positive integers, a partition of the vertex set of a graph $G$ into subsets $X_1,\ldots, X_\ell$, such that vertices in $X_i$ are pairwise at distance greater than $s_i$ for every $i\in\{1,\ldots,\ell\}$, is called an $S$-packing $\ell$-coloring of $G$. The minimum $\ell$ for which $G$ admits an $S$-packing $\ell$-coloring is called the $S$-packing chromatic number of $G$, denoted by $\chi_S(G)$. In this paper, we consider $S$-packing colorings of distance graphs $G(\mathbb{Z},\{k,t\})$, where $k$ and $t$ are positive integers, which are the graphs whose vertex set is $\mathbb{Z}$, and two vertices $x,y\in \mathbb{Z}$ are adjacent whenever $|x-y|\in\{k,t\}$. We complement partial results from two earlier papers, thus determining all values of $\chi_S(G(\mathbb{Z},\{k,t\}))$ when $S$ is any sequence with $s_i\le 2$ for all $i$. In particular, if $S=(1,1,2,2,\ldots)$, then the $S$-packing chromatic number is $2$ if $k+t$ is even, and $4$ otherwise, while if $S=(1,2,2,\ldots)$, then the $S$-packing chromatic number is $5$, unless $\{k,t\}=\{2,3\}$ when it is $6$; when $S=(2,2,2,\ldots)$, the corresponding formula is more complex. 
\end{abstract}


\medskip\noindent
{\bf AMS Subj.\ Class: 05C15, 05C12}

\section{Introduction}
\label{sec:intro}

Given a graph $G$ and a non-decreasing sequence $S=(s_1,s_2,\ldots)$ of positive integers, the mapping $f:V(G) \rightarrow [\ell]=\{1,\ldots,\ell\}$ is an {\em $S$-packing $\ell$-coloring} of $G$ if for any distinct vertices $u,v\in V(G)$ with $f(u)=f(v)=i$, $i\in\{1,\ldots,\ell\}$, the distance between $u$ and $v$ in $G$ is greater than $s_i$.
The smallest $\ell$ such that $G$ has an $S$-packing $\ell$-coloring is the {\em $S$-packing chromatic number} of $G$, denoted by $\chi_S(G)$. This concept was introduced by Goddard, Hedetniemi, Hedetniemi, Harris, and Rall~\cite{goddard-2008}, and was studied in a number of papers; see the recent survey~\cite{BFKR} and the references therein. The main focus of the seminal paper and a number of subsequent papers was on the specific sequence $S=(n)_{n\ge 1}$ in which positive integers appear in the natural order, where the resulting graph invariant is simply called the packing chromatic number~\cite{BKR-07}. Goddard and Xu~\cite{goddard-2012} started consideration of various non-decreasing sequences $S$, and a number of authors followed them. Arguably the most interesting sequences are those that involve only integers $1$ and $2$, since they are in a sense between standard coloring, where $S$ is the constant sequence of $1$s, and $2$-distance coloring, where $S$ is the constant sequence of $2$s (note that a $2$-distance coloring is equivalent to a coloring of the square of a graph and it has been intensively studied in the last decades~\cite{kramer}). In particular, $S$-packing colorings of subcubic graphs were investigated with respect to such sequences $S$~\cite{bkrw-2017b,gt-2016,liu-2020}. Roughly a decade ago, Ekstein et al.~\cite{ekstein-2012} and Togni~\cite{togni-2014} initiated the study of $S$-packing colorings in integer distance graphs, which we present next. Their study was motivated by a series of papers on proper vertex coloring of distance graphs, see e.g. \cite{CCH-1997,DS-2013,DZ-1997,EES-1985} and references therein.

Given a set $D = \{ d_1,\dots, d_h \} $, $h \geq 1$, of positive integers, the {\em (integer) distance graph}, $ G(\mathbb{Z}, D)$, 
is the infinite graph with $\mathbb{Z}$ as the vertex set, while vertices $x$ and $y$ are adjacent if $ |x-y| \in D$. That is, two vertices/integers are adjacent in the graph if their distance in $\mathbb{Z}$ is one of the integers in $\{d_1,\ldots, d_h\}$. We will simplify the notation, and instead of $G(\mathbb{Z}, D)$ write $G(D)$, and for distance sets with two integers we will write $D=\{k,t\}$, and always assume that $k<t$; thus the corresponding distance graph will be written as $G(k,t)$. The packing chromatic numbers of distance graphs $G(k,t)$ were investigated in~\cite{ekstein-2014}. In addition, the $S$-packing colorings of distance graphs $G(k, t$) where $k\in\{1, 2\}$ and $t$ is arbitrary were studied in~\cite{hh-2023,bfk-2021}.
Concerning the sequences $S$ which involve only integers that are not greater than $2$, exact values for $\chi_S(G(1,t))$, where $t\ge 2$, were determined in~\cite{hh-2023, bbs-2019}, while the values $\chi_S(G(2,t))$, where $t\ge 3$, were established in~\cite{bfk-2021}. Hence, for this type of sequences $S$, the $S$-packing chromatic numbers of $G(k,t)$ were left open when $k\ge 3$, and the main goal of this paper is to establish these remaining values. 

In the next section, we establish the notation and give some preliminary observations. In particular, we present two main tools that are used in the proofs. First, we present a representation of the graph $G(k,t)$ in the so-called shifted grid. Second, we introduce color patterns and shift sequences, which enable us a relatively brief presentation of colorings. In Section~\ref{sec:main}, we follow with proving the main results, which are the values of $\chi_S(G(k,t))$, for all $3\le k<t$, and for all possible sequences $S$ involving only integers $1$ and $2$. In Section~\ref{sec:conclude}, we give an overview of results on $S$-packing colorings of the graph $G(k,t)$ by combining the results from~\cite{hh-2023,bfk-2021} with new results from this paper.

\section{Notation and preliminaries}

When presenting sequences $S$, we will often use $i^p$, where $i$ and $p$ are positive integers, as a shortened notation of %and this stands for 
the (sub)sequence $(i,\ldots, i)$, where $i$ appears $p$ times. For instance, $(1^2,2^5)$ stands for the (sub)sequence $(1,1,2,2,2,2,2)$. We may also write $i^\infty$ which coincides with the infinite (sub)sequence $(i,i,i, ...)$. In the case of distinct integers in the sequence $S$, the integer with the power to infinity is the largest among the integers in $S$. For instance, $(1^2,2^\infty)$ presents the sequence with two integers $1$ and all other integers $2$. 

 Note that $ G(k,t) $ is connected if and only if $ \gcd(k,t) = 1 $.
Thus, when determining the $ S $-packing chromatic numbers, we restrict to graphs $G(k,t)$ such that $k$ and $t$ are coprime integers. Note that if $ \gcd (k,t) = g $, then $G(k,t)$ consists of connected components all of which are graphs $G( \frac{k}{g}, \frac{t}{g})$, implying that $ \chi_S(G(k,t)) = \chi_S(G( \frac{k}{g}, \frac{t}{g})) $.

During our study we make use of the following presentation. Notably, a connected distance graph $G(k,t)$ can be represented by the square lattice $\{0,1,\dots,t\} \times \mathbb{Z} $ with vertices given by {\em points}, which are ordered pairs $ (i,j), i \in \{0,1,\dots,t\}, j \in \mathbb{Z} $, such that vertex/point $ (i,j) $ of the grid is a representative of the integer $j \cdot t + i \cdot k $ from $ G(k,t) $. These representatives are unique with the exception of the points whose first coordinate equals to $ 0 $ or $ t $, since a point $(0,j)$ on the grid represents the same integer of $G(k,t)$ as the point $(t,j')$, where $ j'=j-k $.  See Fig.~\ref{fig:obr15}, where ordered pairs in red present points on the grid, while integers in black present the corresponding integers from $\mathbb{Z}$. 

Furthermore, let {\em column $i$} denote the set of vertices $B_i = \{(i,j):\, j\in \mathbb{Z}\} $, where $0\le i\le t$.
As mentioned earlier, integers from $V(G(k,t))$ of the form $jt, j\in \mathbb{Z}$, are represented twice on the grid, notably by a vertex in the column $0$ and a vertex in column $t$. That is, $jt$ is represented by vertex $(0,j)$ as well as vertex $(t,j-k)$.  
For instance, %In particular, 
points $(0,0)$ and $(t,-k)$ represent $0\in V(G(k,t))$. 

\begin{figure}[ht]
	\centering
	\includegraphics[width=10cm]{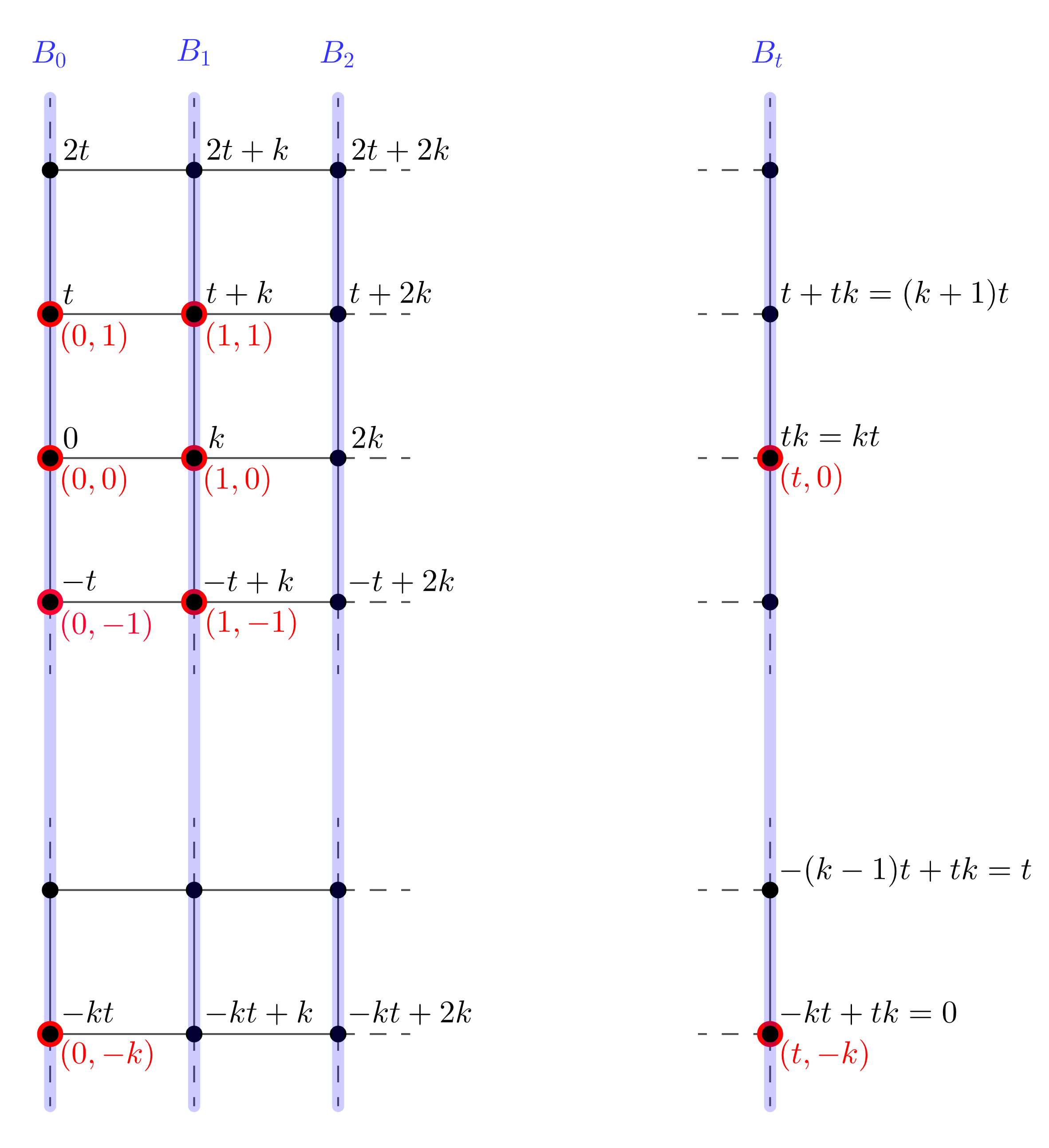}
	\caption{Representation of the distance graph $ G(k,t) $ in the square grid.}
	\label{fig:obr15}
\end{figure}

\subsection{Color patterns}

%In the proofs we will often present an $ S $-packing coloring $ c $ of the graph $ G(k,t) $ given by periodic patterns of length $ d $ applied to the columns $ B_i $.

In this paper, we will often present an $ S $-packing coloring $ c $ by using periodic patterns applied on columns $ B_i $, where $ i \in  \{ 0,\dots,t \}$.
A periodic pattern of length $ d \geq 2 $ is a sequence of colors $ [c_1,\dots,c_d] $ (denoted by square brackets), where the colors $ c_n, n \in [d]$, are not necessarily pairwise distinct.
These colors are given to consecutive vertices within one column and the pattern is applied downwards.
%ONE MAIN PATTERN THAT IS SHIFTED ACCORDING TO SHIFT SEQUENCE AND THEN COPY THE PATTERN TO THE REST OF THE COLUMN
That is, if the coloring $ c$ is using a pattern $ P = [c_1,\dots,c_d] $ in the column $ B_x $ such that $ c(x,y) = c_1 $ for some $ y $, then $ c(x,y-n)=c_{n+1} $ for each $ n \in [d-1] $.
%then $ c(x,y-1) = c_2 $, $ c(x,y-2) = c_3 $ and so on  $ c(x,y-d+1) = c_d $.
This pattern is then periodically copied upwards and downwards to cover all the vertices of $ B_x $. Thus, $ c(x,y+1) = c_d $, $ c(x,y+2) = c_{d-1} $ and so on.

In order for $ c $ to be an $ S $-packing coloring of $G(k,t)$, patterns must be often shifted in consecutive columns.
We define the notion of {\em shift sequence} $ (p_i)_{i=0}^{t-1} $ where $ p_i \in \mathbb{N}_0$. To describe it, we also need the concept of {\em reference point} $(i,j)$ in $B_i$, which is a unique point in each column.
Without loss of generality we may declare the reference point in column $B_0$ to be $(0,0)$.
The integer $ p_i \geq 0 $ represents the value by which the reference point in $ B_{i+1} $ shifts downwards with respect to the reference point in $ B_i $.
That is, if the reference point in $ B_i $ is $( i,j) $, for some $ j\in\mathbb{Z} $, then $ (i+1,j-p_i )$ is the reference point in column $B_{i+1}$. 
(Note that $ p_i=0 $ means that there is no shift.)
To make the elements of the shift sequence correspond to the application of periodic patterns to the columns $ B_i $, we will always assume that the reference point of each column $ B_i $ always receives the first color of the corresponding pattern applied to this column.
That is, if a coloring $ c $ is using the pattern $ P=[c_1,\ldots,c_d] $ in a column $ B_x $ with the reference point $ (x,j) $ then $ c(x,j)=c_1 $.
Thus, the pattern $ P $ used together with the reference point $ (x,j) $ completely determines the coloring of points in $ B_x $.

When shifting the patterns, there are two cases to consider: either columns $ B_i $ and $ B_{i+1} $ are using the same pattern or they are using different patterns. 
First, suppose that a coloring $c$ uses the same pattern $[c_1, \ldots, c_d]$ in columns $B_i$ and $B_{i+1}$. In this case the integer $p_i$ defines the value by which the color $c_1$ in $B_{i+1}$ is shifted with respect to the color $c_1$ in $B_i$. In other words, if $c(i,j)=c_1$ for some $j \in \mathbb{Z}$, then $c(i+1, j-p_i)=c_1$. This implies that $ p_i \neq 0 $ (the pattern must be shifted) since the adjacent vertices must not receive the same color.
Next, let $c$ assign a pattern $[a_1, \ldots, a_{d_1}]$ to the column $B_i$ and a pattern $[b_1, \ldots, b_{d_2}]$ to the column $B_{i+1}$.
In this case, the integer $p_i$ defines the value by which the color $b_1$ given to the reference point of $B_{i+1}$ is shifted with respect to the color $a_1$ given to the reference point of $B_i$. This means that if $c(i,j)=a_1$, where $(i,j)$ is the 
reference point in $B_i$, then $c(i+1, j-p_i)=b_1$ (note that $(i+1, j-p_i)$ is the reference point in $B_{i+1}$). 
%If $ d_1 \neq d_2 $ (the patterns are of different length), the use of reference points in unavoidable.

Since the points $ (0,j) $ represent the same vertices as points $ (t,j-k) $, the relation $ c(0,j) = c(t,j-k) $ must hold for each $ j \in \mathbb{Z} $.
Thus, if $ c $ is an $ S $-packing coloring of $ G(k,t) $, then columns $ B_0 $ and $ B_t $ must use the same periodic pattern and $ \sum_{i=0}^{t-1} p_i \equiv k \pmod d $, where $ d $ is the length of the pattern used on $ B_0 $.
Note that it is sufficient to consider only shifts $ p_i < d_{\max} $ where $ d_{\max} $ denotes the length of the longest pattern.

In this paper, all colorings given by periodic patterns and shift sequences will either use one or two different patterns.
If the second option occurs, it is necessary to specify which column receives which pattern.
Let $ c $ be an $ S $-packing coloring of $ G(k,t) $ given by the shift  sequence $ (p_i)_{i=0}^{t-1}  $ and two periodic patterns $ P_1 = [a_1,\dots,a_{d_1}] $ and $ P_2 = [b_1,\dots,b_{d_2}] $ such that pattern $ P_2 $ is used for example on columns $ B_1 $  and $ B_3 $, and the rest of the columns obtain pattern $ P_1 $.
We will abbreviate such description as:
$$
[a_1,\dots,a_{d_1}]_{p_0} [b_1,\dots,b_{d_2}]_{p_1} [a_1,\dots,a_{d_1}]_{p_2} [b_1,\dots,b_{d_2}]_{p_3} [a_1,\dots,a_{d_1}]_{p_{4\rightarrow t-1}} [a_1,\dots,a_{d_1}],
$$
provided that the shift sequence has $p_4=\ldots =p_{t-1}$.

For a better understanding of the concepts, we next provide a more specific example. Let $ c $ be an $ (1^6) $-coloring of $ G(k,t) $ given by $ P_1 = [1,2,3] $, $ P_2 = [4,5,6] $ and the shift sequence $ (0,1,1,0,1^{t-4}) $ where the notation $ 1^{t-4} $ stands for the (sub)sequence $(1,\dots, 1)$, where $ 1 $ appears $ (t-4) $ times. Again, the pattern $ P_2 $ is applied in columns $ B_1 $ and $ B_3 $, while the rest of the columns use the pattern $ P_1 $.
Thus, the $ (1^6) $-coloring $ c $ is given by:
$$
[1,2,3]_{p_0=0} [4,5,6]_{p_1=1} [1,2,3]_{p_2=1} [4,5,6]_{p_3=0} [1,2,3]_{p_{4 \rightarrow t-1}=1} [1,2,3].
$$
Figure~\ref{fig2} demonstrates how the reference point (marked in red) is shifted with respect to the given shift sequence and how the patterns are applied.

\begin{figure}[ht]
	\centering
	\includegraphics[width=10cm]{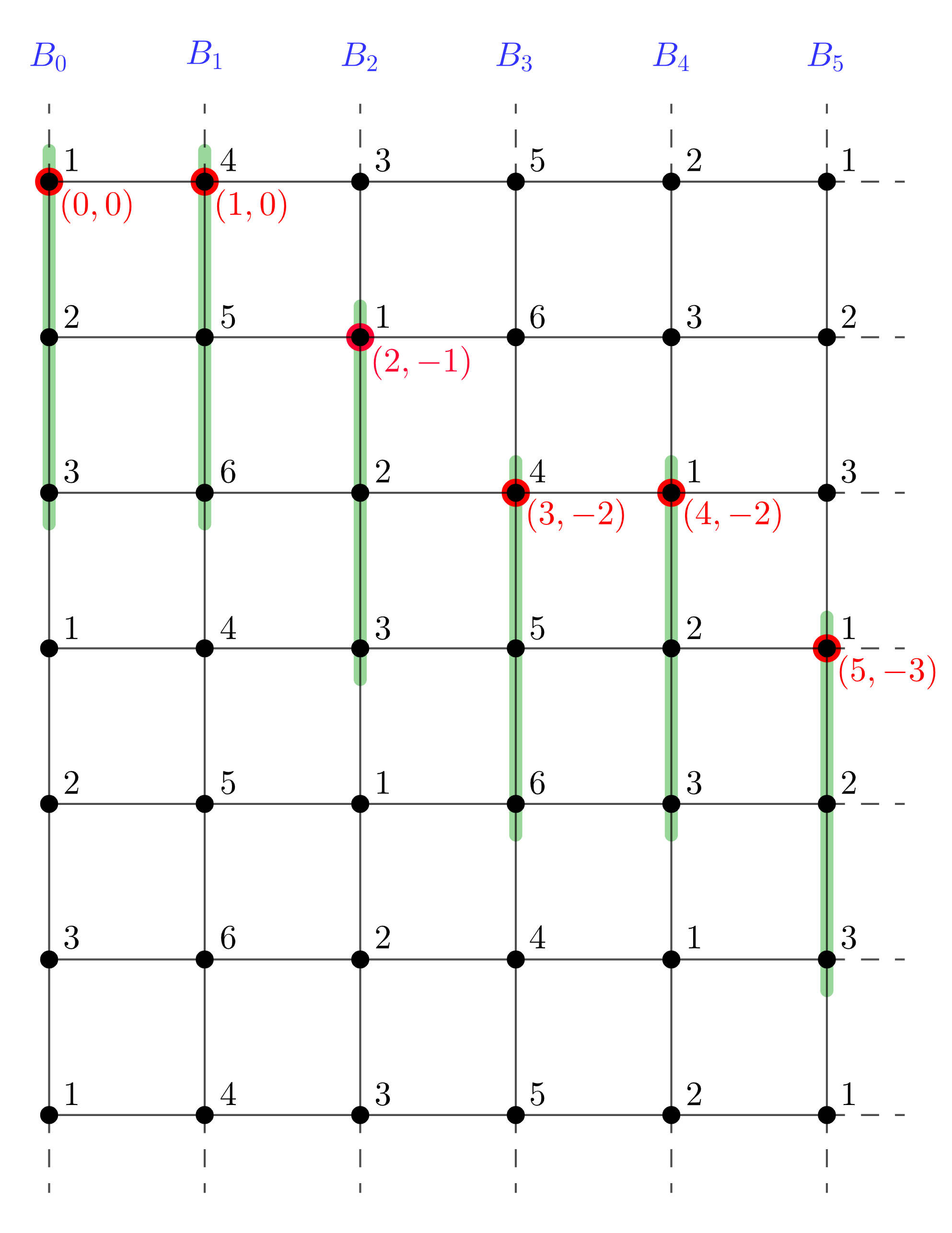}
	\caption{$ (1^6)$-coloring of $ G(k,t) $ given by patterns $ [1,2,3] $ and $ [4,5,6] $, and the shift sequence $ (0,1,1,0,1^{t-4}). $}
	\label{fig2}
\end{figure} 

To verify that $c$ is an $S$-packing coloring of $ G(k,t) $ we will use equivalent conditions, which can be derived from the above notation. Notably, $c$ is an $S$-packing coloring of $ G(k,t)$  if and only if every two vertices with the same color are at sufficient distance, columns $ B_0 $ and $ B_t $ obtain the same pattern and $ \sum_{i=0}^{t-1} p_i \equiv k \ (\mathrm{mod} \ d) $, where $d$ is the length of the pattern in $B_0$. The first among these three conditions requires the following verification. Note that for every two points $(x,y)$ and $(u,v)$ their distance in $G(k,t)$ equals $\min\{|x-u|+|y-v|,x+(t-u)+|(y-k)-v|\}$. Hence, if $c(x,y)=i=c(u,v)$, then
$$\min\{|x-u|+|y-v|,x+(t-u)+|(y-k)-v|\}>s_i.$$
In particular, if $s_i=1$, then it is sufficient to verify that two vertices with color $i$ are not adjacent in the grid. Similarly, if $s_i=2$, then two vertices with color $i$ must not be adjacent in the grid, must not have a common neighbor in the grid, and $c(1,j)=i$ implies $c(t-1,j-k)\ne i$  and vice versa.

\section{$ S $-packing colorings of graphs $ G(k,t) $}
\label{sec:main}

In this section, we present values of $\chi_S(G(k,t))$, where $3\le k<t$ are positive integers, for all possible infinite sequences $S$ whose elements are in $\{1,2\}$. 

For completeness of our study of $S$-packing colorings of graphs $G(k,t)$, we first recall the known results about the standard chromatic number. That is, we consider the sequence $S=(1^\infty)$. Concerning the chromatic number of distance graphs $G(D)$, Walther proved the general bound $\chi(G(D))\le |D|+1$; see~\cite{bbs-2019,wal-1990}. Since $|D|=2$ in our case, we infer that $\chi(G(k,t))\in \{2,3\}$ depending on whether $G(k,t)$ is bipartite or not. Notably, for $3\le k<t$, we derive
$$ \chi(G(k,t))= \left\{
	\begin{array}{ll}
		2; \ k+t \textit{ even,}\\
		3; \ k+t \textit{ odd. }
	\end{array}\right. $$
	
Therefore, whenever a sequence $S$ contains three $1$s, the above result may be applied. In the next subsection, we deal with the three remaining subcases of sequences $S$ depending on the number of $1$s.

\subsection{$ S = (1,1,2,2,2\dots) $}
\label{ss:111}

\begin{theorem}
	If $ G(k,t) $ is the distance graph, where $ k,t $ are coprime positive integers such that $ 3 \leq k < t $, and $ S = (1,1,2^\infty) $, then
	$$ \chi_S(G(k,t))= \left\{
	\begin{array}{ll}
		2; \ k+t \textit{ even},\\
		4; \ k+t \textit{ odd}.
	\end{array}\right. $$
\end{theorem}
\begin{proof}
If $ k+t $ is even, $ G(k,t) $ is bipartite, hence $\chi_S(G(k,t))=2$.
In the rest of the proof, we assume $k+t$ is odd, which implies $\chi_S(G(k,t))\ge 3$. 

Suppose that $\chi_S(G(k,t))=3$, and let $c:V(G(k,t))\rightarrow [3]$ be a $ (1,1,2) $-coloring of $G(k,t)$. Suppose that there is a vertex $ (x,y) \in V(G(k,t)) $ in column $ B_x$, where $x \notin \{0,t\}$, with $ c(x,y)=3 $. For this choice of $ x $, the neighborhood of $ (x,y) $ is the set $ N((x,y)) = \{(x,y+1),(x-1,y),(x+1,y),(x,y-1)\} $.

Note that all vertices in $ N((x,y)) $ have to receive either color $ 1 $ or $ 2 $, since every two vertices with color $ 3 $ must be at the distance at least $ 3 $.
Let $ c(x,y+1)=a $, where $ a\in \{1,2\} $, and let $ \{a,b\}=\{1,2\} $.
Hence, we derive the following chain of implications:
$$
\begin{array}{lll}
 c(x,y+1)=a \Rightarrow c(x+1,y+1)=b \Rightarrow c(x+1,y)=a \Rightarrow  c(x+1,y-1)=b \\ \Rightarrow    c(x,y-1)=a\Rightarrow c(x-1,y-1)=b \Rightarrow c(x-1,y)=a \Rightarrow c(x-1,y+1)=b. \\
\end{array}
$$
We infer that, for every vertex $ (x,y) $ with $ c(x,y)=3 $, all neighbors receive the same color.
(If $x\in \{0,t\}$, we get the same conclusion either by checking the neighborhood of $(x,y)$ as above, or by noting that columns $ B_0 $ and $ B_t $ represent the same yet shifted column by which the coloring can be reassigned so that the corresponding point with color $3$ is in one of columns between $1$ and $t-1$.)
Hence, the coloring $ c':V(G(k,t))\rightarrow [2] $ obtained from $ c $ by recoloring every vertex $ (i,j) $ with $ c(i,j)=3 $ by using the color in $\{1,2\}$ which does not appear in its neighborhood with respect to coloring $ c $ is a $ (1,1) $-coloring of $ G(k,t) $.
Thus, $ G(k,t) $ is bipartite, which is a contradiction with the assumption that $ k+t $ is odd.
We derive that $ \chi_S(G(k,t))\geq 4 $. 

For the proof of the upper bound, we present a $ (1,1,2,2) $-coloring $ c $ given by the shift sequence $ (p_i)_{i=0}^{t-1} = (0,2,0,1^{t-3}) $ and two periodic patterns $ [1,2] $ and $ [3,4,2,1] $ such that the pattern $ [3,4,2,1] $ is used on columns $ B_1 $ and $ B_2 $, and rest of the columns obtain pattern $ [1,2] $.
That is, $ c $ is defined by:
$$
[1,2]_{p_0=0} [3,4,2,1]_{p_1=2} [3,4,2,1]_{p_2=0} [1,2]_{p_{3 \rightarrow t-1}=1} [1,2].
$$
Note that $ c $ is a $ (1,1,2,2) $-coloring of $ G(k,t) $ if and only if three conditions hold: every two vertices with the same color are at sufficient distance, columns $ B_0 $ and $ B_t $ must obtain the same pattern and $ \sum_{i=0}^{t-1} p_i \equiv k \ (\mathrm{mod} \ 2) $.
The matrix below demonstrates the presented coloring $ c $ of $ G(k,t) $ in the first five columns (note that bold integers indicate the location of reference points):
$$
\begin{tabular}{c}
	\vdots \hspace{6pt} \vdots \hspace{6pt} \vdots \hspace{6pt} \vdots \hspace{6pt} \vdots \\
	\textbf{1} \ \textbf{3} \ 2 \ 1 \ 2 \\
	2 \ 4 \ 1 \ 2 \ 1 \\
	1 \ 2 \ \textbf{3} \ \textbf{1} \ 2 \\
	2 \ 1 \ 4 \ 2 \ \textbf{1} \\
	1 \ 3 \ 2 \ 1 \ 2 \\
	2 \ 4 \ 1 \ 2 \ 1 \\
	1 \ 2 \ 3 \ 1 \ 2 \\
	2 \ 1 \ 4 \ 2 \ 1 \\
	\vdots \hspace{6pt} \vdots \hspace{6pt} \vdots \hspace{6pt} \vdots \hspace{6pt} \vdots \\
\end{tabular}
$$
It is easy to verify that all vertices with the same color are at sufficient distance and due to $ t \geq 4 $, columns $ B_0 $ and $ B_t $ always use the same pattern $ [1,2] $.
Thus, the first and the second condition hold.
%What remains is to verify the third condition $ \sum_{i=0}^{t-1} p_i \equiv k \ (\mathrm{mod} \ 2) $.
%Let $ t \equiv \ell \ (\mathrm{mod} \ 2) $, where $ \ell \in \{0,1\} $.
From the following sum we obtain:
$$ \sum_{i=0}^{t-1} p_i = 2 + (t-3) = t-1 \equiv \left\{
\begin{array}{ll}
	0 \ (\mathrm{mod} \ 2); \ t \textit{ odd},\\
	1 \ (\mathrm{mod} \ 2); \ t \textit{ even}.
\end{array}\right. $$
Since $ k $ and $ t $ are of opposite parity, the condition $ \sum_{i=0}^{t-1} p_i \equiv k \ (\mathrm{mod} \ 2) $ also holds.

Therefore, $ \chi_S(G(k,t)) = 4 $ for $ k+t $ odd. \qed
\end{proof}

\subsection{$ S = (1,2,2,2,\dots) $}
\label{ss:1122}

\begin{theorem}
	If $ G(k,t) $ is the distance graph, where $k,t$ are coprime positive integers such that $3 \leq k < t $, and $ S = (1,2^{\infty}) $, then $ \chi_S(G(k,t)) = 5 $.
\end{theorem}

\begin{proof}
	Due to the representation of $ G(k,t) $ as the (shifted) square grid $\{0,1,\dots,t\} \times \mathbb{Z} $, when determining the lower bound of $ \chi_S(G(k,t)) $, we can use results known for the infinite grid $ \mathbb{Z}^2 $. %with vertices given by points $ (i,j), i,j \in \mathbb{Z} $.
	Goddard and Xu~\cite{goddard2014} proved that $ \chi_S(\mathbb{Z}^2) = 5 $,
	hence $ \chi_S(G(k,t)) \geq 5 $. (Alternatively, the same conclusion is derived when observing a vertex in $G(k,t)$ receiving color $1$, and noting that its neighbors must receive pairwise distinct colors from $\{2,\ldots,5\}$.) 
	
	Next, we determine $ (1,2^4) $-colorings of $ G(k,t) $ with respect to the values $k$ and $t $.
	\begin{enumerate}
		\item Let $ t \geq 12 $.
		In this case, we present six different $(1,2^4)$-colorings $ c_n $, where $ n \in \{0,1,2,3,4,5\} $, which are then applied with respect to $k$ and $t$.
		The first $(1,2^4)$-coloring $ c_0 $ is given by just one periodic pattern $ A = [1,2,3,1,4,5] $ and the shift sequence $ \mathbf{q_0} = (2^t) $, while the rest of $ c_n $ given by the shift sequences $ \mathbf{q_n}, n \in \{1,2,3,4,5\} $, also use another pattern $ B = [4,1,5,2,1,3] $ in addition to pattern $ A $. More precisely,
		$$
		\begin{array}{l}
			%\mathbf{q_0} = (2^t), \\
			\mathbf{q_1} = (0,5,2^{t-2}), \\
			\mathbf{q_2} = ((0,5)^2,2^{t-4}), \\
			\mathbf{q_3} = ((0,5)^3,2^{t-6}), \\
			\mathbf{q_4} = ((0,5)^4,2^{t-8}), \\
			\mathbf{q_5} = ((0,5)^5,2^{t-10}), \\
		\end{array}
		$$
		and 
		$$
		\begin{array}{l}
			c_1: A_{p_0=0} B_{p_1=5} A_{p_{2 \rightarrow t-1}=2} A, \\
			c_2: A_{p_0=0} B_{p_1=5} A_{p_2=0} B_{p_3=5} A_{p_{4 \rightarrow t-1}=2} A, \\
			c_3: A_{p_0=0} B_{p_1=5}A_{p_2=0} B_{p_3=5} A_{p_4=0} B_{p_5=5} A_{p_{6 \rightarrow t-1}=2} A, \\
			c_4: A_{p_0=0} B_{p_1=5} A_{p_2=0} B_{p_3=5} A_{p_4=0} B_{p_5=5} A_{p_6=0} B_{p_7=5} A_{p_{8 \rightarrow t-1}=2} A, \\
			c_5: A_{p_0=0} B_{p_1=5} A_{p_2=0} B_{p_3=5} A_{p_4=0} B_{p_5=5} A_{p_6=0} B_{p_7=5} A_{p_8=0} B_{p_9=5} A_{p_{10 \rightarrow t-1}=2} A. \\
		\end{array}
		$$
		Recall that $ c_n $ is a $ (1,2^4) $-coloring of $ G(k,t) $ if and only if three conditions hold: every two vertices with the same color are at sufficient distance, columns $ B_0 $ and $ B_t $ obtain the same pattern, and $ \sum_{i=0}^{t-1} p_i \equiv k \ (\mathrm{mod} \ 6) $.
		From the definition of colorings $ c_n $ we immediately see that both columns $ B_0 $ and $ B_t $ obtain the pattern $ A $, hence the second condition holds.
		
		To verify the first condition, we consider the possible pattern layouts in three consecutive columns, $B_i,B_{i+1},B_{i+2}$, using the patterns $ A $ and $ B $. Since $B_0$ and $B_t$ represent the same, yet shifted column, the case when $i=t-1$ is interpreted as $B_{t-1},B_t,B_1$.  
		We have five possibilities: $AAA$,  $BAA$, $ABA$, $AAB$ and $BAB$. Note that when $AAA$ is used,  in all presented sequences the two shifts are $p_i=2=p_{i+1}$, and it is easy to verify that the points with the same colors are at sufficient distances. This correspondence will be presented as:  
		$$AAA\longleftrightarrow 2,2.$$
Similarly, we have the following correspondences between pattern layouts and shifts, which can be derived from the definitions of sequences $\mathbf{q_n}$:
$$
\begin{array}{c}
BAA\longleftrightarrow 5,2, \\
ABA\longleftrightarrow 0,5, \\
AAB\longleftrightarrow 2,0, \\
BAB\longleftrightarrow 5,0.\\
\end{array}
$$		
(Note that $AAB$ appears when considering the columns $B_{t-1}$, $B_t$ and $B_1$.) 
The corresponding matrices for each of the five cases are placed below.

$$
\begin{array}{lllll}
     \begin{array}{ccc} A& \hspace*{-6pt}A& \hspace*{-6pt}A \\ 
     		\vdots & \hspace*{-6pt} \vdots & \hspace*{-6pt} \vdots \\
     		\mathbf{1} & \hspace*{-6pt} 4 & \hspace*{-6pt} 3 \\
     		2 & \hspace*{-6pt} 5 & \hspace*{-6pt} 1 \\
     		3 &  \hspace*{-6pt}\mathbf{1} & \hspace*{-6pt} 4\\
     		1 & \hspace*{-6pt} 2 & \hspace*{-6pt} 5 \\
     		4 & \hspace*{-6pt} 3 & \hspace*{-6pt} \mathbf{1} \\
     		5 & \hspace*{-6pt} 1 & \hspace*{-6pt} 2 \\
     		1 & \hspace*{-6pt} 4 & \hspace*{-6pt} 3 \\
     		2 & \hspace*{-6pt} 5 & \hspace*{-6pt} 1 \\
     		\vdots & \hspace*{-6pt} \vdots & \hspace*{-6pt} \vdots      		
     \end{array}	
     & \quad
	 \begin{array}{ccc} A&\hspace*{-6pt} A&\hspace*{-6pt} B \\ 
	 		\vdots & \hspace*{-6pt} \vdots & \hspace*{-6pt} \vdots \\
     		\mathbf{1} & \hspace*{-6pt} 4 & \hspace*{-6pt} 1 \\
     		2 & \hspace*{-6pt} 5 & \hspace*{-6pt} 3 \\
     		3 &  \hspace*{-6pt} \mathbf{1} & \hspace*{-6pt} \mathbf{4}\\
     		1 & \hspace*{-6pt} 2 & \hspace*{-6pt} 1 \\
     		4 & \hspace*{-6pt} 3 & \hspace*{-6pt} 5 \\
     		5 & \hspace*{-6pt} 1 & \hspace*{-6pt} 2 \\
     		1 & \hspace*{-6pt} 4 & \hspace*{-6pt} 1 \\
     		2 & \hspace*{-6pt} 5 & \hspace*{-6pt} 3 \\
     		\vdots & \vdots & \vdots      		
	 \end{array}
     & \quad
	 \begin{array}{ccc} A&\hspace*{-6pt} B&\hspace*{-6pt} A \\ 
	 		\vdots & \hspace*{-6pt} \vdots & \hspace*{-6pt} \vdots \\
			\mathbf{1} & \hspace*{-6pt} \mathbf{4} & \hspace*{-6pt} 2 \\
			2 & \hspace*{-6pt} 1 & \hspace*{-6pt} 3 \\
			3 & \hspace*{-6pt} 5 & \hspace*{-6pt} 1 \\
			1 & \hspace*{-6pt} 2 & \hspace*{-6pt} 4 \\
			4 & \hspace*{-6pt} 1 & \hspace*{-6pt} 5 \\
			5 & \hspace*{-6pt} 3 & \hspace*{-6pt} \mathbf{1} \\
			1 & \hspace*{-6pt} 4 & \hspace*{-6pt} 2 \\
			2 & \hspace*{-6pt} 1 & \hspace*{-6pt} 3 \\
     		\vdots & \hspace*{-6pt} \vdots & \hspace*{-6pt} \vdots      		
	 \end{array}
     & \quad
	 \begin{array}{ccc} B&\hspace*{-6pt} A&\hspace*{-6pt} A \\ 
	 		\vdots & \hspace*{-6pt} \vdots & \hspace*{-6pt} \vdots \\
	 		\mathbf{4} & \hspace*{-6pt} 2 & \hspace*{-6pt} 5 \\
	 		1 & \hspace*{-6pt} 3 & \hspace*{-6pt} 1 \\
	 		5 & \hspace*{-6pt} 1 & \hspace*{-6pt} 2 \\
	 		2 & \hspace*{-6pt} 4 & \hspace*{-6pt} 3 \\
	 		1 & \hspace*{-6pt} 5 & \hspace*{-6pt} 1 \\
	 		3 & \hspace*{-6pt} \mathbf{1} & \hspace*{-6pt} 4 \\
	 		4 & \hspace*{-6pt} 2 & \hspace*{-6pt} 5 \\
	 		1 & \hspace*{-6pt} 3 & \hspace*{-6pt} \mathbf{1} \\
	 		\vdots & \hspace*{-6pt} \vdots & \hspace*{-6pt} \vdots
	 \end{array}
     & \quad
	 \begin{array}{ccc} B&\hspace*{-6pt} A&\hspace*{-6pt} B \\ 
	 		\vdots & \hspace*{-6pt} \vdots & \hspace*{-6pt} \vdots \\
	 		\mathbf{4} & \hspace*{-6pt} 2 & \hspace*{-6pt} 1 \\
	 		1 & \hspace*{-6pt} 3 & \hspace*{-6pt} 5 \\
	 		5 & \hspace*{-6pt} 1 & \hspace*{-6pt} 2 \\
	 		2 & \hspace*{-6pt} 4 & \hspace*{-6pt} 1 \\
	 		1 & \hspace*{-6pt} 5 & \hspace*{-6pt} 3 \\
	 		3 & \hspace*{-6pt} \mathbf{1} & \hspace*{-6pt} \mathbf{4} \\
	 		4 & \hspace*{-6pt} 2 & \hspace*{-6pt} 1 \\
	 		1 & \hspace*{-6pt} 3 & \hspace*{-6pt} 5 \\
	 		\vdots & \hspace*{-6pt} \vdots & \hspace*{-6pt} \vdots 		
	 \end{array}
\end{array}
$$

Again, in all cases it is easy to verify that the points with the same colors are at sufficient distances. In this way, the first condition is also verified. 

What remains is to use the third condition $ \sum_{i=0}^{t-1} p_i \equiv k \ (\mathrm{mod} \ 6) $ to determine which sequence $ \mathbf{q_n} $ is suitable for $ G(k,t) $ with respect to values $ k,t $.
		Let $ t \equiv \ell \ (\mathrm{mod} \ 6) $ where $ \ell \in \{0,1,2,3,4,5\} $.
		For each $ \mathbf{q_n} $ we calculate the value of $ \sum_{i=0}^{t-1} p_i $:
%		$$
%		\begin{array}{l}
%			\mathbf{q_0}: \sum_{i=0}^{t-1} p_i = 2t \equiv 2\ell \ (\mathrm{mod} \ 6), \\
%			\\
%			\mathbf{q_1}: \sum_{i=0}^{t-1} p_i = 5 + 2(t-2) = 2t + 1 \equiv (2\ell+1) \ (\mathrm{mod} \ 6), \\
%			\\
%			\mathbf{q_2}: \sum_{i=0}^{t-1} p_i = 5 + 5 + 2(t-4) = 2t + 2 \equiv (2\ell+2) \ (\mathrm{mod} \ 6), \\
%			\\
%			\mathbf{q_3}: \sum_{i=0}^{t-1} p_i = 5 + 5 + 5 + 2(t-6) = 2t + 3 \equiv (2\ell+3) \ (\mathrm{mod} \ 6), \\
%			\\
%			\mathbf{q_4}: \sum_{i=0}^{t-1} p_i = 5 + 5 + 5 + 5 + 2(t-8) = 2t + 4 \equiv (2\ell+4) \ (\mathrm{mod} \ 6), \\
%			\\
%			\mathbf{q_5}: \sum_{i=0}^{t-1} p_i = 5 + 5 + 5 + 5 + 5 + 2(t-10) = 2t + 5 \equiv (2\ell+5) \ (\mathrm{mod} \ 6). \\
%		\end{array}
%		$$
$$
		\mathbf{q_n}: \sum_{i=0}^{t-1} p_i = 5n + 2(t-2n) = 2t + n. % \equiv (2\ell+n) \ (\mathrm{mod} \ 6).
	$$
		Due to appropriate numbering of sequences we obtain $ \sum_{i=0}^{t-1} p_i \equiv (2\ell+n) \ (\mathrm{mod} \ 6) $ for each $ \mathbf{q_n} $.
		In order to determine which $ c_n $ gives us the $ (1,2^4) $-coloring of $ G(k,t) $ for fixed values $ k,t $, let $ k \equiv m \ (\mathrm{mod} \ 6) $ where $ m \in \{0,1,2,3,4,5\} $.
		Using the condition $ \sum_{i=0}^{t-1} p_i \equiv k \ (\mathrm{mod} \ 6) $ we derive:
		$$
		\begin{array}{c}
			(2\ell+n) \ (\mathrm{mod} \ 6) = m \, \Longrightarrow \,
			n = (m-2\ell) \ (\mathrm{mod} \ 6).
		\end{array}
		$$
		Thus, if $ k \equiv m \ (\mathrm{mod} \ 6) $ and $  t \equiv \ell \ (\mathrm{mod} \ 6) $ then the $ (1,2^4) $-coloring of $ G(k,t) $ is given by $ c_n $ using the sequence $ \mathbf{q_n} $ such that $ n = (m-2\ell) \ (\mathrm{mod} \ 6) $.
		
		\item Let $ t = 11 $, hence $ k \in \{3,4,5,6,7,8,9,10\} $.
		For $ G (3,11) $ and $ G(8,11) $ we present a $ (1,2^4) $-coloring given by the periodic pattern $ [1,2,3,4,5] $ and the shift sequence $ (p_i)_{i=0}^{10} = (3^{11}) $.
		The matrix below demonstrates the presented $ (1,2^4) $-coloring of these graphs:
		$$
		\begin{tabular}{c}
			\vdots \hspace{6pt} \vdots \hspace{6pt} \vdots \hspace{6pt} \vdots \hspace{6pt} \vdots \hspace{6pt} \vdots \hspace{6pt} \vdots \hspace{6pt} \vdots \hspace{6pt} \vdots \hspace{6pt} \vdots \hspace{6pt} \vdots \hspace{6pt} \vdots \\
			\textbf{1} \ 3 \ 5 \ 2 \ 4 \ 1 \ 3 \ 5 \ 2 \ 4 \ 1 \ 3 \\
			2 \ 4 \ 1 \ 3 \ 5 \ 2 \ 4 \ 1 \ 3 \ 5 \ 2 \ 4 \\
			3 \ 5 \ 2 \ 4 \ 1 \ 3 \ 5 \ 2 \ 4 \ 1 \ 3 \ 5 \\
			4 \ \textbf{1} \ 3 \ 5 \ 2 \ 4 \ 1 \ 3 \ 5 \ 2 \ 4 \ 1 \\		
			5 \ 2 \ 4 \ 1 \ 3 \ 5 \ 2 \ 4 \ 1 \ 3 \ 5 \ 2 \\
			1 \ 3 \ 5 \ 2 \ 4 \ 1 \ 3 \ 5 \ 2 \ 4 \ 1 \ 3 \\
			2 \ 4 \ \textbf{1} \ 3 \ 5 \ 2 \ 4 \ 1 \ 3 \ 5 \ 2 \ 4 \\
			3 \ 5 \ 2 \ 4 \ 1 \ 3 \ 5 \ 2 \ 4 \ 1 \ 3 \ 5 \\
			4 \ 1 \ 3 \ 5 \ 2 \ 4 \ 1 \ 3 \ 5 \ 2 \ 4 \ 1 \\
			\vdots \hspace{6pt} \vdots \hspace{6pt} \vdots \hspace{6pt} \vdots \hspace{6pt} \vdots \hspace{6pt} \vdots \hspace{6pt} \vdots \hspace{6pt} \vdots \hspace{6pt} \vdots \hspace{6pt} \vdots \hspace{6pt} \vdots \hspace{6pt} \vdots \\
		\end{tabular}
		$$
		It is easy to verify that all vertices with the same color are at a sufficient distance and the condition $ \sum_{i=0}^{10} p_i \equiv 3 \ (\mathrm{mod} \ 5) $ holds for both graphs.
		
		For $ G(4,11) $, $ G(5,11) $, $ G(6,11) $, $ G(7,11) $, $G(9,11)$ and $ G(10,11) $ we use $ (1,2^4) $-colorings $ c_0 $, $ c_1 $, $ c_2 $, $ c_3 $, $c_5$ and $ c_0 $ from case $ 1 $, respectively.
		
%		For $ G(9,11) $ we present an $ (1,2^4) $-coloring given by the periodic pattern $ C = [1,2,1,3,1,4,1,5] $ and the shift sequence $ (p_i)_{i=0}^{10} = (3^{11}) $.
%		Again, the matrix below demonstrates presented $ (1,2^4) $-coloring of this graph:
%		$$
%		\begin{tabular}{c}
%			\vdots \hspace{6pt} \vdots \hspace{6pt} \vdots \hspace{6pt} \vdots \hspace{6pt} \vdots \hspace{6pt} \vdots \hspace{6pt} \vdots \hspace{6pt} \vdots \hspace{6pt} \vdots \hspace{6pt} \vdots \hspace{6pt} \vdots \hspace{6pt} \vdots \\
%			\textbf{1} \ 4 \ 1 \ 5 \ 1 \ 2 \ 1 \ 3 \ 1 \ 4 \ 1 \ 5 \\
%			2 \ 1 \ 3 \ 1 \ 4 \ 1 \ 5 \ 1 \ 2 \ 1 \ 3 \ 1 \\
%			1 \ 5 \ 1 \ 2 \ 1 \ 3 \ 1 \ 4 \ 1 \ 5 \ 1 \ 2 \\
%			3 \ \textbf{1} \ 4 \ 1 \ 5 \ 1 \ 2 \ 1 \ 3 \ 1 \ 4 \ 1 \\		
%			1 \ 2 \ 1 \ 3 \ 1 \ 4 \ 1 \ 5 \ 1 \ 2 \ 1 \ 3 \\
%			4 \ 1 \ 5 \ 1 \ 2 \ 1 \ 3 \ 1 \ 4 \ 1 \ 5 \ 1 \\
%			1 \ 3 \ \textbf{1} \ 4 \ 1 \ 5 \ 1 \ 2 \ 1 \ 3 \ 1 \ 4 \\
%			5 \ 1 \ 2 \ 1 \ 3 \ 1 \ 4 \ 1 \ 5 \ 1 \ 2 \ 1 \\
%			1 \ 4 \ 1 \ 5 \ 1 \ 2 \ 1 \ 3 \ 1 \ 4 \ 1 \ 5 \\
%			2 \ 1 \ 3 \ \textbf{1} \ 4 \ 1 \ 5 \ 1 \ 2 \ 1 \ 3 \ 1 \\
%			\vdots \hspace{6pt} \vdots \hspace{6pt} \vdots \hspace{6pt} \vdots \hspace{6pt} \vdots \hspace{6pt} \vdots \hspace{6pt} \vdots \hspace{6pt} \vdots \hspace{6pt} \vdots \hspace{6pt} \vdots \hspace{6pt} \vdots \hspace{6pt} \vdots \\
%		\end{tabular}
%		$$
%		It is easy to verify that all vertices with the same color are at a sufficient distance and $ \sum_{i=0}^{10} p_i \equiv 1 \ (\mathrm{mod} \ 8) $, hence the condition $ \sum_{i=0}^{10} p_i \equiv k \ (\mathrm{mod} \ 8) $ holds.
		
		\item Let $ t = 10 $, hence $ k \in \{3,7,9\} $.
		For both $ G(3,10) $ and $ G(9,10) $ we use $ (1,2^4) $-coloring $ c_1 $ from case $ 1 $.
		
		For $ G(7,10) $ we present a $ (1,2^4) $-coloring given by two periodic patterns $ C = [1,2,1,3,1,4,1,5] $ and $ D = [4,3,5,4,2,5,3,2] $, and the shift sequence $ (p_i)_{i=0}^9 = ((0,1)^3,3^4) $ such that:
		$$
		C_{p_0=0} D_{p_1=1} C_{p_2=0} D_{p_3=1} C_{p_4=0} D_{p_5=1} C_{p_{6 \rightarrow 9}=3} C.
		$$
		The matrix below demonstrates the presented $ (1,2^4) $-coloring of this graph:
		$$
		\begin{tabular}{c}
			\vdots \hspace{6pt} \vdots \hspace{6pt} \vdots \hspace{6pt} \vdots \hspace{6pt} \vdots \hspace{6pt} \vdots \hspace{6pt} \vdots \hspace{6pt} \vdots \hspace{6pt} \vdots \hspace{6pt} \vdots \hspace{6pt} \vdots \\
			\textbf{1} \ \textbf{4} \ 5 \ 2 \ 1 \ 3 \ 4 \ 1 \ 5 \ 1 \ 2 \\
			2 \ 3 \ \textbf{1} \ \textbf{4} \ 5 \ 2 \ 1 \ 3 \ 1 \ 4 \ 1 \\
			1 \ 5 \ 2 \ 3 \ \textbf{1} \ \textbf{4} \ 5 \ 1 \ 2 \ 1 \ 3 \\
			3 \ 4 \ 1 \ 5 \ 2 \ 3 \ \textbf{1} \ 4 \ 1 \ 5 \ 1 \\		
			1 \ 2 \ 3 \ 4 \ 1 \ 5 \ 2 \ 1 \ 3 \ 1 \ 4 \\
			4 \ 5 \ 1 \ 2 \ 3 \ 4 \ 1 \ 5 \ 1 \ 2 \ 1 \\
			1 \ 3 \ 4 \ 5 \ 1 \ 2 \ 3 \ \textbf{1} \ 4 \ 1 \ 5 \\
			5 \ 2 \ 1 \ 3 \ 4 \ 5 \ 1 \ 2 \ 1 \ 3  \ 1 \\
			\vdots \hspace{6pt} \vdots \hspace{6pt} \vdots \hspace{6pt} \vdots \hspace{6pt} \vdots \hspace{6pt} \vdots \hspace{6pt} \vdots \hspace{6pt} \vdots \hspace{6pt} \vdots \hspace{6pt} \vdots \hspace{6pt} \vdots \\
		\end{tabular}
		$$
		It is easy to verify that all vertices with the same color are at a sufficient distance and the condition $ \sum_{i=0}^{9} p_i \equiv 7 \ (\mathrm{mod} \ 8) $ holds.
		
		\item Let $ t = 9 $, hence $ k \in \{4,5,7,8\} $.
		For $ G(4,9) $ we present a $ (1,2^4) $-coloring given by two periodic patterns $ C = [1,2,1,3,1,4,1,5] $ and $ D = [4,3,5,4,2,5,3,2] $, and the shift sequence $ (p_i)_{i=0}^8 = ((0,1)^3,3^3) $ such that:
		$$
		C_{p_0=0} D_{p_1=1} C_{p_2=0} D_{p_3=1} C_{p_4=0} D_{p_5=1} C_{p_{6 \rightarrow 8}=3} C.
		$$
		The matrix below demonstrates the presented $ (1,2^4) $-coloring of this graph:
		$$
		\begin{tabular}{c}
			\vdots \hspace{6pt} \vdots \hspace{6pt} \vdots \hspace{6pt} \vdots \hspace{6pt} \vdots \hspace{6pt} \vdots \hspace{6pt} \vdots \hspace{6pt} \vdots \hspace{6pt} \vdots \hspace{6pt} \vdots \\
			\textbf{1} \ \textbf{4} \ 5 \ 2 \ 1 \ 3 \ 4 \ 1 \ 5 \ 1 \\
			2 \ 3 \ \textbf{1} \ \textbf{4} \ 5 \ 2 \ 1 \ 3 \ 1 \ 4 \\
			1 \ 5 \ 2 \ 3 \ \textbf{1} \ \textbf{4} \ 5 \ 1 \ 2 \ 1 \\
			3 \ 4 \ 1 \ 5 \ 2 \ 3 \ \textbf{1} \ 4 \ 1 \ 5 \\		
			1 \ 2 \ 3 \ 4 \ 1 \ 5 \ 2 \ 1 \ 3 \ 1 \\
			4 \ 5 \ 1 \ 2 \ 3 \ 4 \ 1 \ 5 \ 1 \ 2 \\
			1 \ 3 \ 4 \ 5 \ 1 \ 2 \ 3 \ \textbf{1} \ 4 \ 1 \\
			5 \ 2 \ 1 \ 3 \ 4 \ 5 \ 1 \ 2 \ 1 \ 3 \\
			\vdots \hspace{6pt} \vdots \hspace{6pt} \vdots \hspace{6pt} \vdots \hspace{6pt} \vdots \hspace{6pt} \vdots \hspace{6pt} \vdots \hspace{6pt} \vdots \hspace{6pt} \vdots \hspace{6pt} \vdots \\
		\end{tabular}
		$$
		It is easy to verify that all vertices with the same color are at a sufficient distance and the condition $ \sum_{i=0}^{8} p_i \equiv 4 \ (\mathrm{mod} \ 8) $ holds.
		
		For $ G(5,9) $ we present a $ (1,2^4) $-coloring given by the periodic pattern $ C = [1,2,1,3,1,4,1,5] $ and the shift sequence $ (p_i)_{i=0}^8 = (5^9) $.
		Again, the matrix below demonstrates the presented $ (1,2^4) $-coloring of this graph:
		$$
		\begin{tabular}{c}
			\vdots \hspace{6pt} \vdots \hspace{6pt} \vdots \hspace{6pt} \vdots \hspace{6pt} \vdots \hspace{6pt} \vdots \hspace{6pt} \vdots \hspace{6pt} \vdots \hspace{6pt} \vdots \hspace{6pt} \vdots \\
			\textbf{1} \ 3 \ 1 \ 2 \ 1 \ 5 \ 1 \ 4 \ 1 \ 3 \\
			2 \ 1 \ 5 \ 1 \ 4 \ 1 \ 3 \ 1 \ 2 \ 1 \\
			1 \ 4 \ 1 \ 3 \ 1 \ 2 \ 1 \ 5 \ 1 \ 4 \\
			3 \ 1 \ 2 \ 1 \ 5 \ 1 \ 4 \ 1 \ 3 \ 1 \\		
			1 \ 5 \ 1 \ 4 \ 1 \ 3 \ 1 \ 2 \ 1 \ 5 \\
			4 \ \textbf{1} \ 3 \ 1 \ 2 \ 1 \ 5 \ 1 \ 4 \ 1 \\
			1 \ 2 \ 1 \ 5 \ 1 \ 4 \ 1 \ 3 \ 1 \ 2 \\
			5 \ 1 \ 4 \ 1 \ 3 \ 1 \ 2 \ 1 \ 5 \ 1 \\
			\vdots \hspace{6pt} \vdots \hspace{6pt} \vdots \hspace{6pt} \vdots \hspace{6pt} \vdots \hspace{6pt} \vdots \hspace{6pt} \vdots \hspace{6pt} \vdots \hspace{6pt} \vdots \hspace{6pt} \vdots \\
		\end{tabular}
		$$
		It is easy to verify that all vertices with the same color are at a sufficient distance and the condition $ \sum_{i=0}^{8} p_i \equiv 5 \ (\mathrm{mod} \ 8) $ holds.
		
		For $ G(7,9) $ and $ G(8,9) $ we use $ (1,2^4) $-colorings $ c_1 $ and $ c_2 $ from case $ 1 $, respectively.
		
		\item Let $ t = 8 $, hence $ k \in \{3,5,7\} $.
		For $ G(3,8) $ we present a $ (1,2^4) $-coloring given by two periodic patterns $ C = [1,2,1,3,1,4,1,5] $ and $ D = [4,3,5,4,2,5,3,2] $, and the shift sequence $ (p_i)_{i=0}^7 = (0,1,3^6) $ such that:
		$$
		C_{p_0=0} D_{p_1=1} C_{p_{2 \rightarrow 7}=3} C.
		$$
		The matrix below demonstrates the presented $ (1,2^4) $-coloring of this graph:
		$$
		\begin{tabular}{c}
			\vdots \hspace{6pt} \vdots \hspace{6pt} \vdots \hspace{6pt} \vdots \hspace{6pt} \vdots \hspace{6pt} \vdots \hspace{6pt} \vdots \hspace{6pt} \vdots \hspace{6pt} \vdots \\
			\textbf{1} \ \textbf{4} \ 5 \ 1 \ 2 \ 1 \ 3 \ 1 \ 4 \\
			2 \ 3 \ \textbf{1} \ 4 \ 1 \ 5 \ 1 \ 2 \ 1 \\
			1 \ 5 \ 2 \ 1 \ 3 \ 1 \ 4 \ 1 \ 5 \\
			3 \ 4 \ 1 \ 5 \ 1 \ 2 \ 1 \ 3 \ 1 \\		
			1 \ 2 \ 3 \ \textbf{1} \ 4 \ 1 \ 5 \ 1 \ 2 \\
			4 \ 5 \ 1 \ 2 \ 1 \ 3 \ 1 \ 4 \ 1 \\
			1 \ 3 \ 4 \ 1 \ 5 \ 1 \ 2 \ 1 \ 3 \\
			5 \ 2 \ 1 \ 3 \ \textbf{1} \ 4 \ 1 \ 5 \ 1 \\
			\vdots \hspace{6pt} \vdots \hspace{6pt} \vdots \hspace{6pt} \vdots \hspace{6pt} \vdots \hspace{6pt} \vdots \hspace{6pt} \vdots \hspace{6pt} \vdots \hspace{6pt} \vdots \\
		\end{tabular}
		$$
		It is easy to verify that all vertices with the same color are at a sufficient distance and the condition $ \sum_{i=0}^{7} p_i \equiv 3 \ (\mathrm{mod} \ 8) $ holds.
		
		For $ G(5,8) $ and $ G(7,8) $ we use $ (1,2^4) $-colorings $ c_1 $ and $ c_3 $ from case $ 1 $, respectively.
				
		\item Let $ t = 7 $, hence $ k \in \{3,4,5,6\} $.
		For $ G(3,7) $, $ G(4,7) $ and $ G(5,7) $ we use $ (1,2^4) $-colorings $ c_1 $, $ c_2 $ and $ c_3 $ from case $ 1 $, respectively.
				
		For $ G (6,7) $ we present a $ (1,2^4) $-coloring given by the periodic pattern $ [1,2,3,4,5] $ and the shift sequence $ (p_i)_{i=0}^6 = (3^7) $.
		The matrix below demonstrates the presented $ (1,2^4) $-coloring of this graph:
		$$
		\begin{tabular}{c}
			\vdots \hspace{6pt} \vdots \hspace{6pt} \vdots \hspace{6pt} \vdots \hspace{6pt} \vdots \hspace{6pt} \vdots \hspace{6pt} \vdots \hspace{6pt} \vdots \\
			\textbf{1} \ 3 \ 5 \ 2 \ 4 \ 1 \ 3 \ 5 \\
			2 \ 4 \ 1 \ 3 \ 5 \ 2 \ 4 \ 1 \\
			3 \ 5 \ 2 \ 4 \ 1 \ 3 \ 5 \ 2 \\
			4 \ \textbf{1} \ 3 \ 5 \ 2 \ 4 \ 1 \ 3 \\		
			5 \ 2 \ 4 \ 1 \ 3 \ 5 \ 2 \ 4 \\
			1 \ 3 \ 5 \ 2 \ 4 \ 1 \ 3 \ 5 \\
			2 \ 4 \ \textbf{1} \ 3 \ 5 \ 2 \ 4 \ 1 \\
			3 \ 5 \ 2 \ 4 \ 1 \ 3 \ 5 \ 2 \\
			\vdots \hspace{6pt} \vdots \hspace{6pt} \vdots \hspace{6pt} \vdots \hspace{6pt} \vdots \hspace{6pt} \vdots \hspace{6pt} \vdots \hspace{6pt} \vdots \\
		\end{tabular}
		$$
		It is easy to verify that all vertices with the same color are at a sufficient distance and $ \sum_{i=0}^{6} p_i \equiv 1 \ (\mathrm{mod} \ 5) $, hence the condition $ \sum_{i=0}^{6} p_i \equiv k \ (\mathrm{mod} \ 5) $ holds.
		
		\item Let $ t = 6 $, hence $ k = 5 $.
		For $ G (5,6) $ we present an $ (1,2^4) $-coloring given by two periodic patterns $ C = [1,2,1,3,1,4,1,5] $ and $ D = [4,3,5,4,2,5,3,2] $, and the shift sequence $ (p_i)_{i=0}^5 = (0,1,3^4) $ such that:
		$$
		C_{p_0=0} D_{p_1=1} C_{p_{2 \rightarrow 5}=3} C.
		$$
		The matrix below demonstrates the presented $ (1,2^4) $-coloring of this graph.
		$$
		\begin{tabular}{c}
			\vdots \hspace{6pt} \vdots \hspace{6pt} \vdots \hspace{6pt} \vdots \hspace{6pt} \vdots \hspace{6pt} \vdots \hspace{6pt} \vdots \\
			\textbf{1} \ \textbf{4} \ 5 \ 1 \ 2 \ 1 \ 3 \\
			2 \ 3 \ \textbf{1} \ 4 \ 1 \ 5 \ 1 \\
			1 \ 5 \ 2 \ 1 \ 3 \ 1 \ 4 \\
			3 \ 4 \ 1 \ 5 \ 1 \ 2 \ 1 \\		
			1 \ 2 \ 3 \ \textbf{1} \ 4 \ 1 \ 5 \\
			4 \ 5 \ 1 \ 2 \ 1 \ 3 \ 1 \\
			1 \ 3 \ 4 \ 1 \ 5 \ 1 \ 2 \\
			5 \ 2 \ 1 \ 3 \ \textbf{1} \ 4 \ 1 \\
			\vdots \hspace{6pt} \vdots \hspace{6pt} \vdots \hspace{6pt} \vdots \hspace{6pt} \vdots \hspace{6pt} \vdots \hspace{6pt} \vdots \\
		\end{tabular}
		$$
		It is easy to verify that all vertices with the same color are at a sufficient distance and the condition $ \sum_{i=0}^{4} p_i \equiv 5 \ (\mathrm{mod} \ 8) $ holds.
		
		\item Let $ t = 5 $, hence $ k \in \{3,4\} $.
		For $ G(3,5) $ we present an $ (1,2^4) $-coloring given by the periodic pattern $ [(1,2,1,3)^2,(1,4,1,5)^2] $ and the shift sequence $ (p_i)_{i=0}^4 = (7^5) $.
		The matrix below demonstrates the presented $ (1,2^4) $-coloring of this graph.
		$$
		\begin{tabular}{c}
			\vdots \hspace{6pt} \vdots \hspace{6pt} \vdots \hspace{6pt} \vdots \hspace{6pt} \vdots \hspace{6pt} \vdots \\
			\textbf{1} \ 4 \ 1 \ 5 \ 1 \ 4 \\
			2 \ 1 \ 3 \ 1 \ 2 \ 1 \\
			1 \ 5 \ 1 \ 4 \ 1 \ 5 \\
			3 \ 1 \ 2 \ 1 \ 3 \ 1 \\
			1 \ 4 \ 1 \ 5 \ 1 \ 2 \\
			2 \ 1 \ 3 \ 1 \ 4 \ 1 \\
			1 \ 5 \ 1 \ 2 \ 1 \ 3 \\
			3 \ \textbf{1} \ 4 \ 1 \ 5 \ 1 \\
			1 \ 2 \ 1 \ 3 \ 1 \ 2 \\
			4 \ 1 \ 5 \ 1 \ 4 \ 1 \\
			1 \ 3 \ 1 \ 2 \ 1 \ 3 \\
			5 \ 1 \ 4 \ 1 \ 5 \ 1 \\
			1 \ 2 \ 1 \ 3 \ 1 \ 4 \\
			4 \ 1 \ 5 \ 1 \ 2 \ 1 \\
			1 \ 3 \ \textbf{1} \ 4 \ 1 \ 5 \\
			5 \ 1 \ 2 \ 1 \ 3 \ 1 \\
			\vdots \hspace{6pt} \vdots \hspace{6pt} \vdots \hspace{6pt} \vdots \hspace{6pt} \vdots \hspace{6pt} \vdots \\
		\end{tabular}
		$$
		It is easy to verify that all vertices with the same color are at a sufficient distance and the condition $ \sum_{i=0}^{4} p_i \equiv 3 \ (\mathrm{mod} \ 16) $ holds.		
		
		For $ G(4,5) $ we use coloring $ c_0 $ from case $ 1 $.
		
		\item Let $ t = 4 $, hence $ k = 3 $.
		For $ G(3,4) $ we use coloring $ c_1 $ from case $ 1 $.
	\end{enumerate}

After verifying all the possible cases for $k$ and $t$, we conclude that $\chi_S(G(k,t))=5$, and the proof is complete. \qed
\end{proof}

\subsection{$ S = (2,2,2,\dots) $}
\label{ss:1222}

\begin{lemma} \label{lemma1}
	If $ G(k,t) $ is the distance graph, where $k,t$ are coprime positive integers such that $3 \leq k < t $, and $ S = (2^{\infty}) $, then $ \chi_S(G(k,t)) \leq 6 $.
\end{lemma}

\begin{proof}
	We determine $(2^6)$-colorings of $ G(k,t) $ with respect to the values of $ k,t $.
	
	\begin{enumerate}
		\item Let $ t \geq 6 $.
		In this case, we present six different $(2^6)$-colorings $ c_n $, where $ n \in \{0,1,2,3,4,5\} $, which we apply with respect to $k$ and $t$.
		Each $ c_n $ is given by the periodic pattern $ [1,2,3,4,5,6] $ and a shift sequence $ \mathbf{q_n} = (p_i)_{i=0}^{t-1} $ such that:
		$$
		\begin{array}{l}
			\mathbf{q_0} = (2^t), \\
			\mathbf{q_1} = (3,2^{t-1}), \\
			\mathbf{q_2} = (3,2,3,2^{t-3}), \\
			\mathbf{q_3} = (3,2,3,2,3,2^{t-5}), \\
			\mathbf{q_4} = (3,4,3,2^{t-3}), \\
			\mathbf{q_5} = (3,4,3,2,3,2^{t-5}). \\
		\end{array}
		$$
		Since the colorings are given by the unique pattern, $ c_n $ is a $ (2^6) $-coloring of $ G(k,t) $ if and only if every two vertices with the same color are at the distance at least $3$ and $ \sum_{i=0}^{t-1} p_i \equiv k \ (\mathrm{mod} \ 6) $. 
		To verify the first condition, we observe a coloring $ c: V(\mathbb{Z}^2) \rightarrow [6] $ of the infinite grid $\mathbb{Z}^2 $ using the pattern $ [1,2,3,4,5,6] $.
		Since $ c $ is $ (2^6) $-coloring, the following must hold for each $ i,j \in \mathbb{Z} $:
		$$
		\begin{array}{l}
			c(i,j) \neq c(i+1,j), \\
			c(i,j) \neq c(i+1,j \pm 1), \\
			c(i,j) \neq c(i+2,j).
		\end{array}
		$$
		Thus, the periodic pattern $ [1,2,3,4,5,6] $ in column $ B_{i+1} $ can be shifted by $ 2, 3 $ or $ 4 $ with respect to the column $ B_i $.
		In addition, the pattern in column $ B_{i+2} $ can be shifted by $ 1, 2, 3, 4 $ or $ 5 $ with respect to the column $ B_i $.
		Thus, for the graph $ G(k,t) $ we obtain $p_i\in\{2,3,4\}$, and
		$$
		\begin{array}{ccc}
			p_i = 2 \Rightarrow p_{i+1} \neq 4, \\
			p_i = 3 \Rightarrow p_{i+1} \neq 3, \\
			p_i = 4 \Rightarrow p_{i+1} \neq 2, \\
		\end{array}
		$$
		holds for all $ i \in \{0,\dots,t-2\} $, and, furthermore,
		$$
		\begin{array}{ccc}
			p_0 = 2 \Rightarrow p_{t-1} \neq 4, \\
			p_0 = 3 \Rightarrow p_{t-1} \neq 3, \\
			p_0 = 4 \Rightarrow p_{t-1} \neq 2, \\
		\end{array}
		$$
		because $ B_0 $ and  $ B_t $ represents the same yet shifted column.
		
		Next, we use the second condition $ \sum_{i=0}^{t-1} p_i \equiv k \ (\mathrm{mod} \ 6) $ to determine which sequence $ \mathbf{q_n} $ is suitable for $ G(k,t) $ with respect to values $ k,t $.
		Let $ t \equiv \ell \ (\mathrm{mod} \ 6) $ where $ \ell \in \{0,1,2,3,4,5\} $.
		For each $ \mathbf{q_n} $ we calculate the value of $ \sum_{i=0}^{t-1} p_i $:
		$$
		\begin{array}{l}
			\mathbf{q_0}: \sum_{i=0}^{t-1} p_i = 2t \equiv 2\ell \ (\mathrm{mod} \ 6), \\
			\\
			\mathbf{q_1}: \sum_{i=0}^{t-1} p_i = 3 + 2(t-1) = 2t + 1 \equiv (2\ell+1) \ (\mathrm{mod} \ 6), \\
			\\
			\mathbf{q_2}: \sum_{i=0}^{t-1} p_i = 3 + 2 + 3 + 2(t-3) = 2t + 2 \equiv (2\ell+2) \ (\mathrm{mod} \ 6), \\
			\\
			\mathbf{q_3}: \sum_{i=0}^{t-1} p_i = 3 + 2 + 3 + 2 + 3 + 2(t-5) = 2t + 3 \equiv (2\ell+3) \ (\mathrm{mod} \ 6), \\
			\\
			\mathbf{q_4}: \sum_{i=0}^{t-1} p_i = 3 + 4 + 3 + 2(t-3) = 2t + 4 \equiv (2\ell+4) \ (\mathrm{mod} \ 6), \\
			\\
			\mathbf{q_5}: \sum_{i=0}^{t-1} p_i = 3 + 4 + 3 + 2 + 3 + 2(t-5) = 2t + 5 \equiv (2\ell+5) \ (\mathrm{mod} \ 6). \\
		\end{array}
		$$
		Thus, for $ \mathbf{q_n} $ we obtain $ \sum_{i=0}^{t-1} p_i \equiv (2\ell+n) \ (\mathrm{mod} \ 6) $.
		In order to determine which $ c_n $ gives us the $ (2^6) $-coloring of $ G(k,t) $ for fixed values $ k,t $, let $ k \equiv m \ (\mathrm{mod} \ 6) $ where $ m \in \{0,1,2,3,4,5\} $.
		From the condition $ \sum_{i=0}^{t-1} p_i \equiv k \ (\mathrm{mod} \ 6) $ we obtain:
		$$
		\begin{array}{c}
			(2\ell+n) \ (\mathrm{mod} \ 6) = m \, \Longrightarrow \,
			n = (m-2\ell) \ (\mathrm{mod} \ 6),
		\end{array}
		$$
		hence if $ k \equiv m \ (\mathrm{mod} \ 6) $ and $  t \equiv \ell \ (\mathrm{mod} \ 6) $ then the $ (2^6) $-coloring of $ G(k,t) $ is given by $ c_n $ using the sequence $ \mathbf{q_n} $ such that $ n = (m-2\ell) \ (\mathrm{mod} \ 6) $.
		
		\item Let $ t = 5 $, hence $ k \in \{3,4\} $.
		For $ G(3,5) $ we present a $ (2^6) $-coloring given by the periodic pattern $ [1,2,3,4,5,1,6,3,2,5,4,6] $ and the shift sequence $ (p_
		i)_{i=0}^4 = (3^5) $.
		The matrix below demonstrates the presented $ (2^6) $-coloring of this graph.
		$$
		\begin{tabular}{c}
			\vdots \hspace{6pt} \vdots \hspace{6pt} \vdots \hspace{6pt} \vdots \hspace{6pt} \vdots \hspace{6pt} \vdots \\
			\textbf{1} \ 5 \ 6 \ 4 \ 1 \ 5 \\
			2 \ 4 \ 3 \ 5 \ 2 \ 4 \\
			3 \ 6 \ 2 \ 1 \ 3 \ 6 \\
			4 \ \textbf{1} \ 5 \ 6 \ 4 \ 1 \\
			5 \ 2 \ 4 \ 3 \ 5 \ 2 \\
			1 \ 3 \ 6 \ 2 \ 1 \ 3 \\
			6 \ 4 \ \textbf{1} \ 5 \ 6 \ 4 \\
			3 \ 5 \ 2 \ 4 \ 3 \ 5 \\
			2 \ 1 \ 3 \ 6 \ 2 \ 1 \\
			5 \ 6 \ 4 \ \textbf{1} \ 5 \ 6 \\
			4 \ 3 \ 5 \ 2 \ 4 \ 3 \\
			6 \ 2 \ 1 \ 3 \ 6 \ 2 \\
			\vdots \hspace{6pt} \vdots \hspace{6pt} \vdots \hspace{6pt} \vdots \hspace{6pt} \vdots \hspace{6pt} \vdots \\
		\end{tabular}
		$$
		It is easy to verify that all vertices with the same color are at a sufficient distance and the condition $ \sum_{i=0}^{4} p_i \equiv 3 \ (\mathrm{mod} \ 12) $ holds.		
		
		For $ G(4,5) $ we use the coloring $ c_0 $ from case $ 1 $. %given by the periodic pattern $ [1,2,3,4,5,6] $ and shift sequence $ \mathbf{q_0} = (2^5) $.
		%It is easy to calculate that:
		%$$ 
		%\sum_{i=0}^{4} p_i = \sum_{i=0}^{4} 2 = 10 \equiv 4 \ (\mathrm{mod} \ 6),
		%$$
		%hence the condition $ \sum_{i=0}^{4} p_i \equiv k \ (\mathrm{mod} \ 6) $ holds.
		
		\item Let $ t = 4 $, hence $ k = 3 $.
		For $ G(3,4) $ we use the coloring $ c_1 $ from case $ 1 $.

	\end{enumerate}
By considering all cases of $k$ and $t$, we obtain $\chi_S(G(k,t))\le 6$, as desired.
\qed
\end{proof}

\begin{theorem}
	If $ G(k,t) $ is the distance graph, where $k,t$ are coprime positive integers such that $3 \leq k < t $, and $ S = (2^{\infty}) $, then
	$$ \chi_S(G(k,t))= \left\{
	\begin{array}{ll}
		5; \ (t \equiv 1,4 \ (\mathrm{mod} \ 5) \textit{ and } k \equiv 2,3 \ (\mathrm{mod} \ 5)) \\
		\ \ \ \textit{ or } (t \equiv 2,3 \ (\mathrm{mod} \ 5) \textit{ and } k \equiv 1,4 \ (\mathrm{mod} \ 5)),\\
		\\
		6; \ \textit{ otherwise}.
	\end{array}\right. $$
\end{theorem}

\begin{proof}
	Consider the infinite grid $ \mathbb{Z}^2 $. %with vertices given by points $ (i,j), i,j \in \mathbb{Z} $. 
	Goddard and Xu~\cite{goddard2014} proved that $ \chi_S(\mathbb{Z}^2) = 5 $,
	hence $ \chi_S(G(k,t)) \geq 5 $.
	Let $ c: V(\mathbb{Z}^2) \rightarrow [5] $ be a $ (2^5) $-coloring of $ \mathbb{Z}^2 $.
	Consider a vertex $ (x,y) \in V(\mathbb{Z}^2) $ with its neighborhood $ N((x,y)) = \{(x,y+1),(x-1,y),(x+1,y),(x,y-1)\} $.
	Each of these $ 5 $ vertices must receive a different color since they are at the distance at most $ 2 $.
	Without loss of generality, let
	$$
	\begin{array}{ccc}
		c(x,y+1) = 1, \
		c(x-1,y) = 2, \
		c(x,y) = 3, \
		c(x+1,y) = 4, \
		c(x,y-1) = 5.
	\end{array}
	$$
	Thus, the vertex $(x+1,y+1) $ can obtain either color $ 2 $ or $ 5 $ with respect to the coloring $ c $.
	We consider both options.
	\begin{enumerate}
		\item Let $ c(x+1,y+1) = 2 $.
		Due to the assignment of colors to vertices in $ N([x,y]) $ and $ (x+1,y+1) $, we derive the following chain of implications:
		$$
		\begin{array}{lll}
		c(x+1,y+1) = 2 \Rightarrow c(x+1,y-1) = 1 \Rightarrow c(x-1,y-1) = 4 \\
		\Rightarrow c(x,y-2) = 2 \Rightarrow
		c(x+1,y-2) = 3 \Rightarrow c(x-1,y-2) = 1 \\
		\Rightarrow c(x,y-3) = 4 \Rightarrow c(x+1,y-3) = 5 \Rightarrow c(x-1,y-3) = 3 \\
		\Rightarrow c(x,y-4) = 1 \Rightarrow \dots
		\end{array}
		 $$
		Thus, the coloring $ c $ forces for each column $ B_i $ the periodic pattern $ [1,3,5,2,4] $.
		Moreover, $ c(i,j) = c(i+1,j-2) $ for all $ i,j \in \mathbb{Z} $.
		\item Let $ c(x+1,y+1) = 5 $.
		Similarly:
		$$
		\begin{array}{lll}
			c(x+1,y+1) = 5 \Rightarrow c(x-1,y+1) = 4 \Rightarrow	c(x-1,y-1) = 1 \\
			\Rightarrow	c(x+1,y-1) = 2 \Rightarrow c(x,y-2) = 4 \Rightarrow c(x-1,y-2) = 3 \\
			\Rightarrow	c(x+1,y-2) = 1 \Rightarrow c(x,y-3) = 2 \Rightarrow	c(x-1,y-3) = 5 \\
			\Rightarrow	c(x+1,y-3) = 3 \Rightarrow	c(x,y-4) = 1 \Rightarrow	\dots
		\end{array}
		$$
		In this case, the coloring $ c $ forces for each column $ B_i $ the periodic pattern $ [1,3,5,4,2] $ and $ c(i,j) = c(i+1,j-3) $ for all $ i,j \in \mathbb{Z} $.
	\end{enumerate}
	We have shown that (up to permutation of colors) there exist only two $2^5$-packing colorings of the square lattice $ \mathbb{Z}^2 $.
	We now apply these two colorings to $ G(k,t) $.
	\begin{enumerate}
		\item Consider the coloring given by the periodic pattern $ [1,3,5,2,4] $ with relation $ c(i,j) = c(i+1,j-2) $.
		Hence, this coloring enforces the constant shift sequence $ (p_i)_{i=0}^{t-1} $ where $ p_i = 2 $ for all $ i \in \{0,\dots,t-1\} $.
		Note that if $ c $ is a $ (2^5) $-coloring of $ G(k,t) $, then $ \sum_{i=0}^{t-1} p_i \equiv k \ (\mathrm{mod} \ 5) $.
		What remains is to determine for which values of $ k,t $ this condition holds.
		
		Let $ t \equiv \ell \ ( \mathrm{mod} \ 5 ) $ where $ \ell \in \{ 0,1,2,3,4 \}. $
		%Thus, the number of elements of $ (p_i)_{i=0}^{t-1} $ gives the remainder $ \ell $ \textcolor{teal}{after integer division by $ 5 $.?}
		From the following sum we derive:
		$$ \sum_{i=0}^{t-1} p_i = \sum_{i=0}^{t-1} 2 = 2t \equiv 2 \ell \ (\mathrm{mod} \ 5), $$
		hence $ k \equiv 2 \ell \ (\mathrm{mod} \ 5) $.
		Note that for $ \ell=0 $ we obtain $ t,k \equiv 0 \ (\mathrm{mod} \ 5) $ which is a contradiction with the proposition $ \gcd(k,t) = 1 $ and thus it is sufficient to consider $ \ell \in \{1,2,3,4\} $.

		\item Consider the coloring given by the periodic pattern $ [1,3,5,4,2] $ with relation $ c(i,j) = c(i+1,j-3) $, which is equivalent to the constant shift sequence $ (p_i)_{i=0}^{t-1} $ where $ p_i = 3 $ for all $ i \in \{0,\dots,t-1\} $.
		We again determine for which $ k,t $ the condition $ \sum_{i=0}^{t-1} p_i \equiv k \ (\mathrm{mod} \ 5) $ holds.
		
		Let $ t \equiv \ell \ ( \mathrm{mod} \ 5 ) $ where $ \ell \in \{ 0,1,2,3,4 \}. $
		Similarly to the previous case:
		$$ \sum_{i=0}^{t-1} p_i = \sum_{i=0}^{t-1} 3 = 3t \equiv 3\ell \ (\mathrm{mod} \ 5), $$
		hence $ k \equiv 3 \ell \ (\mathrm{mod} \ 5) $.
		For $ \ell=0 $ we obtain $ t,k \equiv 0 \ (\mathrm{mod} \ 5) $, contradicting $ \gcd(k,t)=1 $ again, thus we consider $ \ell \in \{1,2,3,4\} $.
	\end{enumerate}
	By calculating the values of $ k,t $ for the given $\ell $, we obtain that $ \chi_S(G(k,t)) = 5 $ if and only if either $ (t \equiv 1,4 \ (\mathrm{mod} \ 5) $ and $ k \equiv 2,3 \ (\mathrm{mod} \ 5)) $ or $ (t \equiv 2,3 \ (\mathrm{mod} \ 5) $ and $ k \equiv 1,4 \ (\mathrm{mod} \ 5)) $.
	From Lemma \ref{lemma1} we derive $ \chi_S(G(k,t)) = 6 $ for the remaining values of $ k,t $. \qed
\end{proof}
 
\section{Concluding remarks}
\label{sec:conclude}

These results presented in this paper complement previously known results from~\cite{bbs-2019, bfk-2021, hh-2023, wal-1990} on $S$-packing colorings of connected distance graphs $G(k,t)$, which satisfy the property that each integer in a sequence $S$ belongs to $\{1,2\}$. Note that such $S$-packing colorings are classical colorings ($S=(1^\infty)$), $2$-distance colorings ($S=(2^\infty)$) and $S$-packing colorings which lie between these two. 

The amalgamation of the results presented herein with those established earlier gives us the $S$-packing chromatic numbers of all connected distance graphs $G(k,t)$, where $k \geq 1$ and $t>k$ are coprime. With respect to the sequence $S$ we summarize all of these results as follows. 
\begin{enumerate}
%\vspace{0.5cm}
\item $S=(1^\infty)$.

$ \chi(G(k,t))= \left\{
	\begin{array}{ll}
		2; \ k+t \textit{ even,}\\
		3; \ k+t \textit{ odd. }
	\end{array}\right. $
	
%\vspace{0.5cm}
\item $S=(1,1, 2^\infty)$.
	
$ \chi_S(G(k,t))= \left\{
	\begin{array}{ll}
		2; \ k+t \textit{ even},\\
		4; \ k+t \textit{ odd}.
	\end{array}\right. $

%\vspace{0.5cm}	
\item $S=(1,2^\infty)$.

$ \chi_S(G(k,t))= \left\{
	\begin{array}{ll}
		5; \ k \neq 2 \textit{ or }  t \neq 3,\\
		6; \ \textit{otherwise}.
	\end{array}\right. $

%\vspace{0.5cm}
\item $S=(2^\infty)$.	

$ \chi_2(G(k,t))= \left\{
	\begin{array}{ll}
		5; \ (t \equiv 1,4 \ (\mathrm{mod} \ 5) \textit{ and } k \equiv 2,3 \ (\mathrm{mod} \ 5)) \\
		\ \ \ \textit{ or } (t \equiv 2,3 \ (\mathrm{mod} \ 5) \textit{ and } k \equiv 1,4 \ (\mathrm{mod} \ 5)),\\
		\\
		7; \ k=2 \textit{ and } t=3,\\
		6; \ \textit{ otherwise}.
	\end{array}\right. $

As for other sequences $S$ that contain elements greater than $2$, investigations of $\chi_S(G(k,t))$ are far from complete. For the sequence $S=(d^\infty)$, $d \geq 3$, lower and upper bounds are known for $\chi_S(G(k,t))$~\cite{bfk-2021}, which in some cases culminate to exact results. For instance, if $t \geq 5$ is an odd integer and $d \geq t-3$, then $\chi_d(G(k,t))=1+t \cdot \left(d - \frac{t-3}{2}\right)$. Similarly, lower and upper bounds are known for the sequence $S=(1,2,3, \ldots)$ which corresponds to the standard packing coloring. This results focus on the distance graph $G(1,t)$~\cite{ekstein-2012,togni-2014}. Additionally, exact results exist for certain sporadic sequences $S$, which came into fruition while studying the $S$-packing chromatic numbers of $G(k,t)$ for sequences that contain only elements from $\{1,2\}$. The $S$-packing coloring of $G(1,t)$ provided in \cite{hh-2023} partitions the color classes in such a way that the vertices of some color classes are farther apart than they need to be. As a consequence, this gives results for the sequences $S$ with larger elements.

There remains ample scope for research concerning the $S$-packing chromatic number of distance graphs $G(k,t)$. In addition, for distance graphs $G(D)$, where $|D|\ge 3$, only a few results on their $S$-packing colorings are known. 

\end{enumerate}	
\section*{Acknowledgements}
B.B., J.F. and M.J. acknowledge the financial support of the Slovenian Research and Innovation Agency (research core funding No.\ P1-0297 and projects N1-0285, J1-3002, and J1-4008).
%K.K. was partially supported by the project GA20-09525S of the Czech Science Foundation. 


\begin{thebibliography}{99}

\bibitem{bbs-2019} B.~Benmedjdoub, I.~Bouchemakh, \'{E}.~Sopena, 2-distance colorings of integer distance graphs, Discuss.\ Math.\ Graph Theory 39 (2019) 589-603.

\bibitem{bfk-2021} B.~Bre\v sar, J.~Ferme, K.~Kamenick\'{a},  
$S$-packing colorings of distance graphs $G(\mathbb{Z},\{2,t\})$, Discrete Appl.\ Math.\ 298 (2021) 143--154.


\bibitem{BFKR} B.~Bre\v sar, J.~Ferme, S.~Klav\v zar and D.F.~Rall,
 A survey on packing colorings, Discuss.\ Math.\ Graph Theory 40 (2020) 923--970.

\bibitem{BKR-07} 
B.~Bre\v sar, S.~Klav\v zar, D.F.~Rall,
    On the packing chromatic number of Cartesian products, hexagonal lattice, and trees,
   Discrete Appl.\ Math.\ 155 (2007) 2303--2311.

 
\bibitem{bkrw-2017b}
  B.~Bre\v sar, S.~Klav\v zar, D.F.~Rall and K.~Wash,
 Packing chromatic number, $(1,1,2,2)$-colorings, and characterizing the Petersen graph,
  Aequationes Math.\ 91 (2017) 169--184.


\bibitem{CCH-1997} J. Chen, G. Chang, K. Huang,  Integral distance graphs, J. Graph Theory 25 (1997) 287-294.
  
\bibitem{DS-2013} 
D. Der-Fen Liu, A. Sutedja, Chromatic number of distance graphs generated by the sets $\{2,3,x,y\}$, Journal Combin. Optim. 25 (2013) 680-693. 

\bibitem{DZ-1997} 
W.A. Deuber, X. Zhu,  The chromatic numbers of distance graphs, Discrete Math. 165/166 (1997) 195-204. 
  
\bibitem{EES-1985} 
R.B. Eggleton, P. Erd\H os, D.K. Skilton,  Colouring the real line, J. Combin. Theory ser. B 39 (1985) 86-100.

\bibitem{ekstein-2012}
 J.~Ekstein, P.~Holub and B.~Lidick\'y,
 Packing chromatic number of distance graphs,
  Discrete Appl.\ Math.\ 160 (2012) 518--524.

\bibitem{ekstein-2014}
  J.~Ekstein, P.~Holub and O.~Togni,
The packing coloring of distance graphs $D(k,t)$,
  Discrete Appl.\ Math.\ 167 (2014) 100--106.
  
\bibitem{gt-2016}
  N.~Gastineau and O.~Togni,
$S$-packing colorings of cubic graphs,
  Discrete Math.\  339 (2016) 2461--2470.


\bibitem{goddard-2008}
  W.~Goddard, S.M.~Hedetniemi, S.T.~Hedetniemi, J.M.~Harris and D.F.~Rall,
 Broadcast chromatic numbers of graphs,
  Ars Combin.\  86 (2008) 33--49.
  
\bibitem{goddard-2012}
  W.~Goddard and H.~Xu,
The $S$-packing chromatic number of a graph,
  Discuss.\ Math.\ Graph Theory 32 (2012) 795--806.
  
  \bibitem{goddard2014}
  W.~Goddard and H.~Xu,
   A note on $S$-packing colorings of lattices,
  Discrete Appl.\ Math.\ 166 (2014) 255--262.
  
  
\bibitem{hh-2023} P.~Holub, J.~Hofman, 
On S-packing colourings of distance graphs $D(1,t)$ and $D(1,2,t)$, Appl.\ Math.\ Comput.\ 447 (2023) Paper No. 127855.

  
\bibitem{kramer}
  F.~Kramer and H.~Kramer,
 A survey on the distance-colouring of graphs,
  Discrete Math.\  308 (2008) 422--426.
  
  
\bibitem{liu-2020}  R.~Liu, X.~Liu, M.~Rolek and G.~Yu, Packing $(1, 1, 2, 2)$-coloring of some subcubic graphs, Discrete Appl.\ Math.\ 283 (2020) 626-630. 
  
\bibitem{togni-2014}
  O.~Togni,
  On packing colorings of distance graphs,
  Discrete Appl.\ Math.\ 167 (2014) 280--289.
  
  \bibitem{wal-1990} H.~Walther, \"{U}ber eine spezielle Klasse unendlicher Graphen, in: Graphentheorie,
K. Wagner and R. Bodendiek (Ed(s)), (Bibl. Inst., Mannheim, 1990) 2, 268--295.


\end{thebibliography}
\end{document}